\documentclass[12pt]{amsart}
\usepackage{geometry, hyperref, verbatim, enumerate,amsthm, amsmath, amssymb, amscd, mathrsfs, color, tikz-cd}

\title{Some results on null ideals of finite rings}

\author{Nicholas J. Werner}
\address{Department of Mathematics, Computer and Information Science, SUNY at Old Westbury, Old Westbury, NY 11568,USA}
\email{wernern@oldwestbury.edu}

\numberwithin{equation}{section}

\theoremstyle{definition}\newtheorem{Def}[equation]{Definition}
\theoremstyle{plain}\newtheorem{Lem}[equation]{Lemma}
\theoremstyle{plain}\newtheorem{Prop}[equation]{Proposition}
\theoremstyle{plain}\newtheorem{Thm}[equation]{Theorem}
\theoremstyle{plain}
\theoremstyle{plain}
\theoremstyle{remark}\newtheorem{Rem}[equation]{Remark}
\theoremstyle{definition}\newtheorem{Ex}[equation]{Example}
\theoremstyle{definition}
\theoremstyle{definition}\newtheorem{Ques}[equation]{Question}
\theoremstyle{definition}

\newcommand{\ee}{\mathfrak{e}}
\newcommand{\ff}{\mathfrak{f}}
\newcommand{\msJ}{\mathscr{J}}

\newcommand{\Ieek}[1]{I_{\ee \ee}^{#1}}
\newcommand{\Iffk}[1]{I_{\ff \ff}^{#1}}
\newcommand{\Iefk}[1]{I_{\ee \ff}^{#1}}
\newcommand{\Ifek}[1]{I_{\ff \ee}^{#1}}

\newcommand{\mcI}{\mathcal{I}}
\newcommand{\mcN}{\mathcal{N}}
\newcommand{\mcR}{\mathcal{R}}

\newcommand{\F}{\mathbb{F}}

\begin{document}

\begin{abstract}
For a finite associative unital ring $R$, the null ideal of $R$ is the collection of polynomials with coefficients from $R$ that send each element of $R$ to zero under evaluation. It was conjectured that the null ideal of $R$ is always a two-sided ideal of its overlying polynomial ring. The conjecture was proved to be false with the construction of a subring of $4 \times 4$ upper triangular matrices over $\F_2$ for which the null ideal is not two-sided. The Jacobson radical of this counterexample ring has nilpotency 4. We prove that if the Jacobson radical of $R$ has nilpotency at most 3, then the null ideal of $R$ is two-sided. By extending the known counterexample ring, for each $n \geq 5$ we present a ring for which the null ideal is not two-sided, and the Jacobson radical has nilpotency $n$.
\end{abstract}

\maketitle

\begin{center}\today\end{center}

\section{Introduction}\label{sec: introduction}
\thispagestyle{empty}

Throughout, $R$ is a finite associative ring with unity and $R[x]$ is the ring of polynomials over $R$ in a central indeterminate $x$. When $R$ is noncommutative, we will follow standard practices as in \cite[\S 16]{Lam} for performing computations in $R[x]$. Since $x$ is central, polynomials can be added and multiplied as in the usual commutative case. Moreover, we will assume that polynomials satisfy right evaluation. So, before $f \in R[x]$ can be evaluated at an element of $R$, $f$ must be written so that powers of $x$ appear to the right of the coefficients. For instance, if $a, b \in R$ and $f(x)=(ax)(bx)$, then we must write $f(x)=abx^2$ before evaluating $f$. 

The purpose of this paper is to study the \textit{null ideal} $\mcN(R)$ of $R$, which is the set of polynomials in $R[x]$ that send each element of $R$ to 0 under evaluation. That is,
\begin{equation*}
\mcN(R) := \{f \in R[x] \mid f(a) = 0 \text{ for all } a \in R\}.
\end{equation*}
Recent studies on null ideals of finite rings include \cite{Rissner2016,Frisch2017,RogersWickham2017,RogersWickham2018,SwartzWerner2023,Werner2022,LeroyETAL2026}. Null ideals are notable for their connections to general polynomial functions on rings (see e.g.\ \cite{Al-EzehETAL2021, Al-Maktry2023}) and rings of integer-valued polynomials (see e.g.\ \cite[Section 2]{Werner2022} and \cite[Chapter XV]{ChabertIVP}). 

When $R$ is a commutative ring, it is trivial to verify that $\mcN(R)$ is an ideal of $R[x]$ (whence the term ``null ideal''), but the properties of $\mcN(R)$ are less clear when $R$ is noncommutative. Given $f(x), g(x) \in R[x]$, denote their product by $(fg)(x)$. The set $\mcN(R)$ is closed under addition, but because $f$ and $g$ may have coefficients that do not commute, it is possible that $(fg)(a) \ne f(a) g(a)$ for some $a \in R$. Thus, if $R$ is noncommutative, then it is not apparent whether or not $\mcN(R)$ is an ideal (left, right, or two-sided) of $R[x]$.

In \cite[Conjecture 3.3]{Werner2014}, it was conjectured that $\mcN(R)$ is a two-sided ideal of $R$ for every finite ring $R$. This conjecture is true for many classes of finite rings, including all finite commutative rings, all semisimple finite rings, and any ring of odd order \cite[Theorem 3.7]{Werner2014}. Nevertheless, \cite[Conjecture 3.3]{Werner2014} was proved false in \cite{Havlovec1}, in which the large language model GPT-5.6 Sol was used to find a finite ring for which the null ideal is not two-sided.

\begin{Ex}\label{ex: Havlovec} \cite{Havlovec1} 
For a prime power $q$, let $\F_q$ be the finite field with $q$ elements. Let $M_n(\F_q)$ be the ring of $n \times n$ matrices with entries from $\F_q$, and let $T_n(\F_q)$ be the ring of $n \times n$ upper triangular matrices with entries from $\F_q$. Define
\begin{equation*}
\mcR := \left\{\begin{pmatrix} A & B\\ 0 & A \end{pmatrix} \bigg| A \in T_2(\F_2), \; B \in M_2(\F_2) \right\} \subseteq M_4(\F_2).
\end{equation*}
It is shown in \cite[Theorem 2.1]{Havlovec1} that $\mcN(\mcR)$ is not a two-sided ideal of $\mcR[x]$.
\end{Ex}

The goal of the present article is to demonstrate that there is a sense in which the ring $\mcR$ from \cite{Havlovec1} is a minimal possible counterexample to \cite[Conjecture 3.3]{Werner2014}. For a ring $R$, let $\msJ(R)$ be the Jacobson radical of $R$. Often, we will set $J = \msJ(R)$. When $R$ is finite, $J$ is a nilpotent ideal of $R$, and the \textit{nilpotency} of $J$ is the smallest positive integer $n$ such that $J^n=\{0\}$. Each element of $\mcR$ is a $4 \times 4$ upper triangular matrix, and one may check that the Jacobson radical of $\mcR$ has nilpotency 4. Our main result is the following.

\begin{Thm}\label{thm: n=2, 3}
If the nilpotency of the Jacobson radical of $R$ is at most 3, then $\mcN(R)$ is a two-sided ideal of $R[x]$.
\end{Thm}

Thus, when characterizing finite rings based on the nilpotency of their Jacobson radical, $\mcR$ is a minimal example of a ring for which the null ideal is not two-sided. Whether $\mcR$ is a ring of minimal \textit{order} with this property is not known, and we leave this as an open problem.

\begin{Ques}\label{ques: minimal order}
The ring $\mcR$ of Example \ref{ex: Havlovec} has order 128. Does there exist a ring $R$ such that $|R| < 128$ and $\mcN(R)$ is not a two-sided ideal of $R[x]$?
\end{Ques}

In Section \ref{sec: prelim}, we establish our basic notations and recall a number of sufficient conditions under which $\mcN(R)$ is two-sided. In particular, the problem can be reduced to the consideration of products $f\ee$, where $f \in \mcN(R)$ and $\ee$ is an idempotent of $R$ lifted from $R/J$. In Section \ref{sec: Ideals of J}, we present methods to produce two-sided ideals of $R$ inside $J$. These ideals are used in Section \ref{sec: null} to prove Theorem \ref{thm: n=2, 3}. We close the paper in Section \ref{sec: counter} with a detailed analysis of the counterexample ring $\mcR$. We provide a list of sufficient conditions on a ring $R$ so that $\mcN(R)$ is not two-sided (see Definition \ref{def: counterex} and Theorem \ref{thm: counterex}). Using these conditions, we define an infinite family of rings (including and generalizing $\mcR$) for which the null ideal is not two-sided.

\section{Preliminaries and prior results}\label{sec: prelim}

As in the introduction, let $R$ be a finite ring with Jacobson radical $J$. Let $\pi:R \to R/J$ be the canonical quotient map. The residue ring $R/J$ is semisimple. Let $t \geq 1$ be the number of simple summands in a direct sum decomposition of $R$. That is,
\begin{equation}\label{eq: R/J}
R/J \cong \bigoplus_{i=1}^t M_{n_i}(F_i),
\end{equation}
where, for each $i$, $F_i$ is a finite field, $n_i \geq 1$, and $M_{n_i}(F_i)$ is the ring of $n_i \times n_i$ matrices with entries from $F_i$. Following \cite[Remark p.\ 116]{McD}, we can lift a system of orthogonal idempotents from $R/J$ to $R$.

\begin{Def}\label{def: E(R)}
For each $1 \leq i \leq t$, fix an idempotent $\ee_i$ of $R$ such that the following properties hold:
\begin{itemize}
\item $1 = \ee_1 + \cdots + \ee_t$,
\item $\ee_i \ee_j = 0$ if and only if $i \ne j$, and
\item $\ee_i \text{ mod } J$ is the multiplicative identity of $M_{n_i}(F_i)$.
\end{itemize}
Let $E(R) = \{\ee_1, \ldots, \ee_t\}$.
\end{Def}

Recall that $\mcN(R) = \{f \in R[x] \mid f(a) = 0 \text{ for all } a \in R\}$. It is clear that $\mcN(R)$ is closed under addition. For polynomials $f$ and $g$ satisfying right evaluation---as is the case in our setting---it is known that if $a \in R$ and $f(a) = 0$, then $(gf)(a)=0$. Thus, $\mcN(R)$ is always a left ideal of $R[x]$. To determine whether $\mcN(R)$ is a right ideal, it suffices to consider products of the form $fa$, where $f \in \mcN(R)$ and $a \in R$. By \cite[Lemma 2.3]{Werner2014}, $\mcN(R)$ is a two-sided ideal if and only if $fa \in \mcN(R)$ for every $f \in \mcN(R)$ and every $a \in R$. Furthermore, \cite[Lemma 3.5]{Werner2014} shows that $fu \in \mcN(R)$ for every $f \in \mcN(R)$ and all units $u \in R$. From this, it follows that if every element of $R$ is a sum of units of $R$, then $\mcN(R)$ is two-sided. Finite rings in which each element is a sum of units can be characterized by examining the summands of $R/J$.

\begin{Prop}\label{prop: sums of units}
\mbox{}
\begin{enumerate}[(1)]
\item \cite[Theorem 4.6]{Stewart1972} 
Each element of $R$ is a sum of units of $R$ if and only if $R/J$ contains no direct summand isomorphic to $\F_2 \oplus \F_2$.
\item Assume that $R/J$ has $m$ direct summands isomorphic to $\F_2$. Let $B = \bigoplus_{i=1}^m \F_2 \subseteq R/J$, so that $R/J \cong B \oplus R'$, where $R'$ has no direct summand isomorphic to $\F_2$. Let $\pi_B: R \to B$ be the canonical quotient map. Then, $a \in R$ is a sum of units if and only if $\pi_B(a) = 0_B$ or $\pi_B(a) = 1_B$.
\end{enumerate}
\end{Prop}
\begin{proof}
(2) $(\Rightarrow)$ Assume that $a \in R$ is a sum of units of $R$. Then, $\pi(a)$ is a sum of units of $R/J$. Since $\F_2$ has a trivial unit group, so does $B$. It follows that $1_B$ and $0_B = 1_B + 1_B$ are the only elements of $B$ that are sums of units of $B$. Thus, either $\pi_B(a) = 0_B$ or $\pi_B(a) = 1_B$.

$(\Leftarrow)$ Assume that $\pi_B(a) = 0_B$ or $\pi_B(a) = 1_B$. By part (1), each element of $R'$ is a sum of units of $R'$. So, there exist units $u_1, \ldots, u_k \in R'$ such that $\pi(a) = \pi_B(a) + \sum_{i=1}^k u_i$. Hence, $\pi(a)$ is a sum of units of $R/J$. Since units can be lifted modulo the Jacobson radical of a ring, $a$ is a sum of units of $R$.
\end{proof}

From Proposition \ref{prop: sums of units}, we see that if a finite ring $R$ is such that $\mcN(R)$ is not two-sided, then $R/J$ must contain a copy of $\F_2 \oplus \F_2$. Indeed, this occurs for the ring $\mcR$ in \cite{Havlovec1}, for which $\mcR/\msJ(\mcR) \cong \F_2 \oplus \F_2$. While this condition is necessary for $\mcN(R)$ to fail to be two-sided, it is not sufficient. For instance, $\mcN(\F_2 \oplus \F_2)$ is two-sided, because $\F_2 \oplus \F_2$ is a commutative ring.

We end this section by showing that to determine whether or not $\mcN(R)$ is two-sided, it suffices to consider right multiplication of polynomials by the idempotents in $E(R)$.

\begin{Lem}\label{lem: idempotents suffice}
The following are equivalent.
\begin{enumerate}[(1)]
\item $\mcN(R)$ is a two-sided of $R[x]$.
\item For all $f \in \mcN(R)$ and all $a \in R$, $fa \in \mcN(R)$.
\item For all $f \in \mcN(R)$ and all $\ee \in E(R)$, $f\ee \in \mcN(R)$.
\end{enumerate}
\end{Lem}
\begin{proof}
The implications $(1) \Rightarrow (2)$ and $(2) \Rightarrow (3)$ are clear, and $(2) \Rightarrow (1)$ was proved in \cite[Lemma 2.3(2)]{Werner2014}. 

$(3) \Rightarrow (2)$ Assume that (3) holds. As in Proposition \ref{prop: sums of units}, decompose $R/J$ as $R/J \cong B \oplus R'$, where $B = \bigoplus_{i=1}^m \F_2 \subseteq R/J$ for some $m \geq 0$ and $R'$ has no direct summand isomorphic to $\F_2$ (if $m=0$, then $B=0$ and $R/J \cong R'$). Index the idempotents in $E(R)$ so that $\ee_1, \ldots, \ee_m$ correspond to the direct summands of $B$.

Let $f \in \mcN(R)$ and $a \in R$. Then, $a=b+r'+\alpha$, where $\alpha \in J$ and $b, r' \in R$ are such that $\pi(b) \in B$ and $\pi(r') \in R'$. Furthermore, we may take $b = \sum_{i=1}^m c_i \ee_i$, where each $c_i \in \{0,1\}$. By assumption, $fb \in \mcN(R)$. By Proposition \ref{prop: sums of units}, $r'+\alpha$ is a sum of units of $R$, so $f \cdot (r'+\alpha) \in \mcN(R)$ by \cite[Lemma 3.5]{Werner2014}. Thus, $fa \in \mcN(R)$.
\end{proof}

\section{Two-sided ideals of \texorpdfstring{$J$}{J}}\label{sec: Ideals of J}

Maintain the notation giving in Section \ref{sec: prelim}. Recall that the nilpotency of $J$ is the smallest positive integer $n$ such that $J^n=\{0\}$. Throughout this section, assume that $J$ has nilpotency $n \geq 2$. In this section, we describe constructions to produce two-sided ideals contained in $J$. All of these ideals involve Peirce decompositions using the system of orthogonal idempotents $E(R)$ from Definition \ref{def: E(R)}.

\begin{Lem}\label{lem: eiRej}
\mbox{}
\begin{enumerate}[(1)]
\item Each $\ee \in E(R)$ is central in $R/J$. Thus, for all $a \in R$ and all $\ee \in E(R)$, there exist $\alpha \in J$ such that $a \ee = \ee a + \alpha$.
\item Let $1 \leq i, j \leq t$ such that $i \ne j$. Then, $\ee_i R \ee_j = \ee_i J \ee_j $.
\item Let $\ee \in E(R)$ and let $\ff = 1 - \ee$. Then, $\ee R \ff = \ee J \ff$ and $\ff R \ee = \ff J \ee$.
\end{enumerate}
\end{Lem}
\begin{proof}
(1) With $R/J$ as in \eqref{eq: R/J}, each $\ee_i$ is central in $M_{n_i}(F_i)$ and annihilates the ring $\bigoplus_{j \ne i} M_{n_j}(F_j)$ on both the left and the right. Then, $\ee_i$ is central in $R/J$. 

(2) By (1), $\ee_i R \ee_j \subseteq (R\ee_i + J)\ee_j$. Since $i \ne j$, $\ee_i \ee_j = 0$. Hence, $(R\ee_i + J)\ee_j = J\ee_j$. Similarly, $\ee_i R \ee_j \subseteq \ee_i J$. So, 
\begin{equation*}
\ee_i R \ee_j \subseteq J\ee_j \cap \ee_i J = \ee_i J \ee_j.
\end{equation*}
Clearly, $\ee_i J \ee_j \subseteq \ee_i R \ee_j$, so in fact $\ee_i R \ee_j = \ee_i J \ee_j$.

(3) This follows from (2).
\end{proof}

\begin{Lem}\label{lem: J^n-1}
\mbox{}
\begin{enumerate}[(1)]
\item Let $1 \leq i, j \leq t$. Then, $\ee_i J^{n-1} \ee_j$ is a two-sided ideal of $R$.
\item Let $\ee \in E(R)$ and let $\ff = 1 - \ee$. Then, $\ee J^{n-1} \ff$ is a two-sided ideal of $R$.
\end{enumerate}
\end{Lem}
\begin{proof}
(1) Let $a \in R$. There exist $s_1, \ldots, s_t \in R$ such that $\pi(a) = \pi(\sum_{i=1}^t s_i \ee_i)$. So, $a = \sum_{i=1}^t s_i \ee_i + \alpha$ for some $\alpha \in J$. Then, $\alpha(\ee_i J^{n-1} \ee_j) = \{0\}$ and $\ee_k \ee_i = 0$ when $k \ne i$, so
\begin{equation*}
a(\ee_i J^{n-1} \ee_j) = s_i\ee_i(\ee_i J^{n-1} \ee_j) = s_i\ee_i J^{n-1} \ee_j.
\end{equation*}
Next, the images of $s_i$ and $\ee_i$ commute in $R/J$, so $s_i \ee_i = \ee_i s_i + \alpha_i$ for some $\alpha_i \in J$. Since $\alpha_i(\ee_i J^{n-1} \ee_j) = \{0\}$, we have
\begin{equation*}
s_i\ee_i J^{n-1} \ee_j = \ee_i s_i J^{n-1} \ee_j \subseteq \ee_i J^{n-1} \ee_j.
\end{equation*}
Thus, $a(\ee_i J^{n-1} \ee_j) \subseteq \ee_i J^{n-1} \ee_j$. The proof that $(\ee_i J^{n-1} \ee_j)a \subseteq \ee_i J^{n-1} \ee_j$ is similar.

(2) This follows from (1), because $\ff$ is a sum of elements of $E(R)$.
\end{proof}

When $k < n-1$, there is no guarantee that $\ee J^k \ff$ is a two-sided ideal of $R$. However, if we include additional direct summands reminiscent of a Peirce decomposition of $J^k$, then we can produce other two-sided ideals contained in $J$.

\begin{Def}\label{def: Ieek ideals}
Assume that $n \geq 2$ and let $1 \leq k \leq n-1$. Let $\ee \in E(R)$ and let $\ff = 1 - \ee$. We define the following four additive subgroups of $J$:
\begin{align*}
\Ieek{(k)} &:= \ee J^k \ee \oplus \ee J^{k+1} \ff \oplus \ff J^{k+1} \ee \oplus \ff J^{k+2} \ff, \\
\Iffk{(k)} &:= \ff J^k \ff \oplus \ee J^{k+1} \ff \oplus \ff J^{k+1} \ee \oplus \ee J^{k+2} \ee, \\
\Iefk{(k)} &:= \ee J^k \ff \oplus \ee J^{k+1} \ee \oplus \ff J^{k+1} \ff \oplus \ff J^{k+2} \ee, \text{ and}\\
\Ifek{(k)} &:= \ff J^k \ee \oplus \ee J^{k+1} \ee \oplus \ff J^{k+1} \ff \oplus \ee J^{k+2} \ff.
\end{align*}

\end{Def}
\begin{Prop}\label{prop: Ieek ideals}
Assume that $n \geq 2$ and let $1 \leq k \leq n-1$. Let $\ee \in E(R)$ and let $\ff = 1 - \ee$.
\begin{enumerate}[(1)]
\item $\Ieek{(k)}$, $\Iffk{(k)}$, $\Iefk{(k)}$, and $\Ifek{(k)}$ are all two-sided ideals of $R$. 
\item $\Ieek{(k)} \cap \Iffk{(k)} \cap \Iefk{(k)} \cap \Ifek{(k)} = J^{k+2}$. 
\item When $n \geq 3$, $\Ieek{(n-2)} \cap \Iffk{(n-2)} \cap \Iefk{(n-2)} \cap \Ifek{(n-2)} = \{0\}$.
\end{enumerate}
\end{Prop}
\begin{proof}
(1) We prove that $\Ieek{(k)}$ is a two-sided ideal of $R$; the proofs for $\Iffk{(k)}$, $\Iefk{(k)}$, and $\Ifek{(k)}$ are similar. First, to show that $R\Ieek{(k)} \subseteq \Ieek{(k)}$, it suffices to have $R\ee\Ieek{(k)} \subseteq \Ieek{(k)}$ and $R\ff\Ieek{(k)} \subseteq \Ieek{(k)}$. Since $\ee\ff = 0$, we have $R\ee\Ieek{(k)} = R\ee(\ee J^k \ee) \oplus R\ee(\ee J^{k+1} \ff)$. Decomposing $R$ as $\ee R \oplus \ff R$ and recalling that $\ee R \ff = \ee J \ff$ and $\ff R \ee = \ff J \ee$ by Lemma \ref{lem: eiRej}, we obtain
\begin{align*}
R\ee\Ieek{(k)}  &= R\ee(\ee J^k \ee) \oplus R\ee(\ee J^{k+1} \ff) \\
&= \ee R\ee J^k \ee \oplus \ff R\ee J^k \ee \oplus \ee R \ee J^{k+1} \ff \oplus \ff R \ee J^{k+1}\ff\\
&\subseteq \ee J^k \ee \oplus \ff J^{k+1} \ee \oplus \ee J^{k+1} \ff \oplus \ff J^{k+2}\ff\\
&= \Ieek{(k)}.
\end{align*}
Likewise,
\begin{align*}
R\ff \Ieek{(k)} &= R \ff (\ff J^{k+1} \ee) \oplus R \ff (\ff J^{k+2} \ff)\\
&= \ee R \ff (\ff J^{k+1} \ee) \oplus \ff R \ff (\ff J^{k+1} \ee) \oplus \ee R \ff (\ff J^{k+2} \ff) \oplus \ff R \ff (\ff J^{k+2} \ff)\\
&\subseteq \ee J^{k+2} \ee \oplus \ff J^{k+1} \ee \oplus \ee J^{k+3} \ff \oplus \ff J^{k+2} \ff\\
&\subseteq \Ieek{(k)}.
\end{align*}

So far, we have shown that $R\Ieek{(k)} \subseteq \Ieek{(k)}$. Analogous arguments will prove that $\Ieek{(k)}R \subseteq \Ieek{(k)}$. Explicitly, one may verify that
\begin{align*}
\Ieek{(k)}\ee R &\subseteq \ee J^k \ee \oplus \ee J^{k+1} \ff \oplus \ff J^{k+1} \ee \oplus \ff J^{k+2} \ff = \Ieek{(k)}, \text{and}\\
\Ieek{(k)}\ff R &\subseteq \ee J^{k+2} \ee \oplus \ee J^{k+1} \ff \oplus \ff J^{k+3} \ee \oplus \ff J^{k+2} \ff \subseteq \Ieek{(k)}.
\end{align*}
Thus, $\Ieek{(k)}$ is a two-sided ideal of $R$.\\

(2) Let $I = \Ieek{(k)} \cap \Iffk{(k)} \cap \Iefk{(k)} \cap \Ifek{(k)}$. Apply a two-sided Peirce decomposition to $I$ to obtain
\begin{equation*}
I := \ee I \ee \oplus \ee I \ff \oplus \ff I \ee \oplus \ff I \ff,
\end{equation*}
where the direct sum is of additive groups. If $\alpha \in I$, then $\ee \alpha \ee \in \ee \Iffk{(k)} \ee = \ee J^{k+2} \ee$, so $\ee I \ee \subseteq \ee J^{k+2} \ee$. The reverse containment holds because $\ee J^{k+2} \ee \subseteq \Iffk{(k)} \subseteq I$, so $\ee J^{k+2} \ee = \ee(\ee J^{k+2} \ee)\ee \subseteq \ee I \ee$. Thus, $\ee I \ee = \ee J^{k+2} \ee$. Similarly, $\ee I \ff = \ee J^{k+2} \ff$, $\ff I \ee = \ff J^{k+2} \ee$, and $\ff I \ff = \ff J^{k+2} \ff$. The result follows by applying a two-sided Peirce decomposition to $J^{k+2}$.\\

(3) This follows from (2) and the fact that $J^n = \{0\}$.
\end{proof}

\begin{Lem}\label{lem: Iee is 0}
Assume $n \geq 2$ and let $1 \leq k \leq n-1$. Let $\ee \in E(R)$ and let $\ff = 1 - \ee$.
\begin{enumerate}[(1)]
\item If $\Ieek{(k)} = \{0\}$ or $\Iefk{(k)} = \{0\}$, then $\ff J^k$ is a two-sided ideal of $R$.
\item If $\Iffk{(k)} = \{0\}$ or $\Ifek{(k)} = \{0\}$, then $\ee J^k$ is a two-sided ideal of $R$.
\end{enumerate}
\end{Lem}
\begin{proof}
We prove (1); the proof of (2) is similar. Let $I = \ff J^k$, which is a right ideal of $R$. We begin by showing that $J^{k+1} \subseteq I$.

Suppose first that $\Ieek{(k)} = \{0\}$. Then, 
\begin{equation*}
\ee J^k \ee = \ee J^{k+1} \ff = \ff J^{k+1} \ee = \ff J^{k+2} \ff = \{0\},
\end{equation*}
which means that $J^{k+1} = \ff J^{k+1} \ff \subseteq I$. Next, consider the case where $\Iefk{(k)} = \{0\}$. This implies that
\begin{equation*}
\ee J^k \ff = \ee J^{k+1} \ee = \ff J^{k+1} \ff = \ff J^{k+2} \ee = \{0\},
\end{equation*}
and hence $J^{k+1} = \ff J^{k+1} \ee \subseteq \ff J^k$. Thus, from the stated assumptions, we can conclude that $J^{k+1} \subseteq I$.

Now, let $\alpha \in J$, and consider $\alpha I$. We have
\begin{equation*}
\alpha I = \alpha \ff J^k \subseteq J^{k+1} \subseteq I.
\end{equation*}
So, $I$ is closed under left multiplication by elements of $J$. Let $r \in R$. Since $\ff$ is central in $R/J$, there exists $\beta \in J$ such that $r\ff = \ff r + \beta$. Note that $\ee I = \ee \ff J^k = \{0\}$. We have
\begin{equation*}
rI = r\ee I + r \ff I = r \ff I = \ff r I + \beta I.
\end{equation*}
Now, $\ff r I = \ff r \ff J^k \subseteq \ff J^k = I$, and $\beta I \subseteq I$ by the previous paragraph. Thus, $rI \subseteq I$, and $I$ is a two-sided ideal of $R$.
\end{proof}

\section{Null ideals}\label{sec: null}

We now return our focus to the null ideal $\mcN(R)$ of $R$. By using Lemma \ref{lem: idempotents suffice} and the two-sided ideals constructed in Definition \ref{def: Ieek ideals}, we will prove that $\mcN(R)$ is a two-sided ideal of $R[x]$ whenever the nilpotency of $J$ is 2 or 3.

\begin{Lem}\label{lem: what fe kills}\cite[Lemma 4.2]{Werner2014}
Let $e$ be any idempotent of $R$ and let $f \in \mcN(R)$. Then, $(fe)(a)=0$ whenever $a \in eR$ or $a \in Re + (1-e)R$.
\end{Lem}

\begin{Lem}\label{lem: poly decomp}
Assume $n \geq 2$ and let $1 \leq k \leq n-1$. Let $\ee \in E(R)$ and let $\ff = 1 - \ee$.
\begin{enumerate}[(1)]
\item Assume $\Ieek{(k)} = \{0\}$ or $\Iefk{(k)} = \{0\}$. Let $a \in R \ff + \ee R$ and let $b \in \ff J^k$. Then, for all $i \geq 1$,
\begin{equation*}
\ff(a+b)^i = \ff a^i + (a+b)^i - a^i.
\end{equation*}

\item Assume $\Iffk{(k)} = \{0\}$ or $\Ifek{(k)} = \{0\}$. Let $a \in R \ee + \ff R$ and let $b \in \ee J^k$. Then, for all $i \geq 1$,
\begin{equation*}
\ee(a+b)^i = \ee a^i + (a+b)^i - a^i.
\end{equation*}
\end{enumerate}
\end{Lem}
\begin{proof}
We prove (1); the proof of (2) is similar. Since either $\Ieek{(k)} = \{0\}$ or $\Iefk{(k)} = \{0\}$, we know that $\ff J^k$ is a two-sided ideal of $R$ by Lemma \ref{lem: Iee is 0}. So, for each $i \geq 1$, there exists $\alpha_i \in \ff J^k$ such that $(a+b)^i = a^i + \alpha_i$. Then,
\begin{equation*}
\ff(a+b)^i = \ff a^i + \ff \alpha_i = \ff a^i + \alpha_i = \ff a^i + (a+b)^i - a^i. \qedhere 
\end{equation*}
\end{proof}

\begin{Prop}\label{prop: poly decomp n=3}
Assume $n \geq 2$. Let $\ee \in E(R)$ and let $\ff = 1 - \ee$. If at least one of $\Ieek{(1)}$, $\Iffk{(1)}$, $\Iefk{(1)}$, or $\Ifek{(1)}$ is the zero ideal, then $f\ee \in \mcN(R)$ for every $f \in \mcN(R)$.
\end{Prop}
\begin{proof}
Let $f \in \mcN(R)$. Note that $f\ee \in \mcN(R)$ if and only if $f\ff \in \mcN(R)$ because $f\ff = f - f\ee$.

Suppose first that either $\Ieek{(1)} = \{0\}$ or $\Iefk{(1)} = \{0\}$. By Lemma \ref{lem: eiRej}, $\ff R \ee = \ff J \ee$, so we may decompose $R$ as
\begin{equation*}
R = R\ff \oplus R\ee = R \ff \oplus \ee R \ee \oplus \ff R \ee = R \ff \oplus \ee R \ee \oplus \ff J \ee.
\end{equation*}
Consequently, given $r \in R$, we may write $r = a + b$, where $a \in R\ff \oplus \ee R \ee$ and $b \in \ff J \ee$. Let $f(x) = \sum_{i \geq 1} c_i x^i$; note that the constant coefficient of $f$ is 0 because $f(0) = 0$. By Lemma \ref{lem: poly decomp}(1), we have
\begin{align*}
(f\ff)(r) &= \sum_i c_i \ff r^i\\
&= \sum_i c_i \ff (a+b)^i\\
&= \sum_i c_i \big(\ff a^i + (a+b)^i - a^i\big)\\
&=(f\ff)(a) + f(r) - f(a).
\end{align*}
By Lemma \ref{lem: what fe kills}, $(f\ff)(a) = 0$, and both $f(r)$ and $f(a)$ are 0 because $f \in \mcN(R)$. Thus, $f\ff \in \mcN(R)$, as desired.

If either $\Iffk{(1)} = \{0\}$ or $\Ifek{(1)} = \{0\}$, then we use the same argument, but switch the roles of $\ee$ and $\ff$. In these cases,  write $R$ as
\begin{equation*}
R = R \ee \oplus \ff R \ff \oplus \ee J \ff,
\end{equation*}
and let $r = a + b$, where $a \in R\ee \oplus \ff R \ff$ and $b \in \ee J \ff$. Applying Lemma \ref{lem: poly decomp}(2) to $(f\ee)(r)$ shows that $f\ee \in \mcN(R)$.
\end{proof}

We now have all of the tools necessary to prove Theorem \ref{thm: n=2, 3}.

\begin{proof}[Proof of Theorem \ref{thm: n=2, 3}]
Let $J$ be the Jacobson radical of $R$ and assume that $J$ has nilpotency $n \leq 3$. If $n=1$, then $R$ is semisimple and $\mcN(R)$ is two-sided by \cite[Theorem 3.7]{Werner2014}. So, assume that $n=2$ or $n=3$. Let $\ee \in E(R)$ and let $\ff = 1 - \ee$. By Proposition \ref{prop: Ieek ideals}, all of $\Ieek{(1)}$, $\Iffk{(1)}$, $\Iefk{(1)}$, and $\Ifek{(1)}$ are two-sided ideals of $R$. For each $I \in \{\Ieek{(1)}, \Iffk{(1)}, \Iefk{(1)}, \Ifek{(1)}\}$, let $N(R,I) = \{f \in R[x] \mid f(R) \subseteq I\}$ and let $\pi_I: R[x] \to (R/I)[x]$ be the canonical quotient map. Then, for each $I$, $N(R,I)$ is the inverse image of $\mcN(R/I)$ under $\pi_I$. Furthermore, by Proposition \ref{prop: poly decomp n=3}, $\pi_I(f\ee) \in \mcN(R/I)$ for each $I$ and for every $f \in \mcN(R)$. Thus, $f\ee \in \mcN(R,I)$ for each $I$ and every $f \in \mcN(R)$. Finally, $\Ieek{(1)} \cap \Iffk{(1)} \cap \Iefk{(1)} \cap \Ifek{(1)} = \{0\}$ by Proposition \ref{prop: Ieek ideals}, so
\begin{equation*}
\mcN(R) = \mcN(R, \Ieek{(1)}) \cap \mcN(R, \Iffk{(1)}) \cap \mcN(R, \Iefk{(1)}) \cap \mcN(R, \Ifek{(1)}),
\end{equation*}
and hence $f\ee \in \mcN(R)$. Since this holds for all $\ee \in E(R)$ and all $f \in \mcN(R)$, $\mcN(R)$ is a two-sided ideal of $R[x]$ by Lemma \ref{lem: idempotents suffice}.
\end{proof}

\section{Counterexamples to the null ideal conjecture}\label{sec: counter}

In the recent preprint \cite{Havlovec1}, a large language model was used to produce a counterexample to \cite[Conjecture 3.3]{Werner2014}, which proposed that $\mcN(R)$ is a two-sided ideal of $R[x]$ for every finite ring $R$. In this section, we examine this counterexample and present some conditions that can be used to produce more general finite rings for which the null ideal is not two-sided.

As in Example \ref{ex: Havlovec}, let
\begin{equation*}
\mcR := \left\{\begin{pmatrix} A & B\\ 0 & A \end{pmatrix} \bigg| A \in T_2(\F_2), \; B \in M_2(\F_2) \right\},
\end{equation*}
which is a finite ring of order 128 with Jacobson radical of nilpotency 4. By \cite[Theorem 2.1]{Havlovec1}, $\mcN(\mcR)$ is not a two-sided ideal of $\mcR[x]$. In \cite{Havlovec1}, this is proved directly by giving a specific polynomial $f \in \mcR[x]$ and a specific element $P \in \mcR$ such that $f \in \mcN(\mcR)$ but $fP \notin \mcN(\mcR)$. In the spirit of the prior sections of this paper, we will describe the failure of $\mcN(\mcR)$ to be two-sided in terms of the Jacobson radical and idempotents of $\mcR$. This more axiomatic approach allows us to construct other rings that are also counterexamples to \cite[Conjecture 3.3]{Werner2014}, and affords a more uniform treatment of the counterexamples. 

\begin{Def}\label{def: counterex}
Let $R$ be a finite ring with Jacobson radical $J$. We call $R$ a \textit{counterexample ring} if $R$ satisfies all of the properties below.
\begin{enumerate}[(1)]
\item $R/J \cong \F_2 \oplus \F_2$.
\item $J$ has nilpotency $n \geq 4$.
\item There exists $r \in R$ such that $(r^2+r)^{n-1} \ne 0$.
\item There exists $\ee \in E(R)$ such that the following three conditions hold (as usual, we let $\ff = 1 - \ee$):
\begin{itemize}
\item[(4a)] $J^{n-1} = \ee J^{n-1} \ff$.
\item[(4b)] $\ee J^{n-2} \ff = \ee J^{n-1} \ff$.
\item[(4c)] As additive groups, $\ee J^2 \ff$ is an index 2 subgroup of $\ee J \ff$. That is, $|\ee J \ff / \ee J^2 \ff| = 2$. In particular, this means that $\ee J \ff \setminus \ee J^2 \ff$ is nonempty.
\end{itemize}
\end{enumerate}
\end{Def}

The conditions listed in Definition \ref{def: counterex} are inspired by properties of $\mcR$, and any ring $R$ satisfying these conditions will be such that $\mcN(R)$ is not two-sided.

\begin{Lem}\label{lem: counterex lem}
Let $R$ be a counterexample ring.
\begin{enumerate}[(1)]
\item Let $y \in J$ and $z \in J^{n-2}$. Then, $yz = (\ee y \ff) z \ff$ and $zy = \ee z(\ee y \ff)$.

\item Let $\beta \in \ee J \ff \setminus \ee J^2 \ff$. Then, for all $\alpha \in J$, $\beta \alpha^{n-2} = \alpha^{n-1}$.
\end{enumerate}
\end{Lem}
\begin{proof}
(1) Since $yz \in J^{n-1}$, we have $yz = \ee yz \ff$ by \ref{def: counterex}(4a). So,
\begin{equation*}
yz = \ee y (\ee + \ff) z \ff = \ee y (\ee z \ff) + (\ee y \ff) z \ff.
\end{equation*}
By \ref{def: counterex}(4b), $\ee z \ff \in J^{n-1}$, so $y (\ee z \ff) \in J^n=\{0\}$. Thus, $yz = (\ee y \ff) z \ff$. Similar steps will show that $zy = \ee z(\ee y \ff)$.\\

(2) Let $\alpha \in J$. We consider two cases depending on $\ee \alpha \ff$. First, assume that $\ee \alpha \ff \in \ee J^2 \ff$. Then, by part (1), 
\begin{equation*}
\alpha^{n-1} = (\ee \alpha \ff) \alpha^{n-2} \ff \subseteq J^n = \{0\},
\end{equation*}
and for $\beta \alpha^{n-2}$ we have
\begin{equation*}
\beta \alpha^{n-2} = (\beta \alpha^{n-3})\alpha = \ee(\beta \alpha^{n-3})(\ee \alpha \ff) \subseteq J^n = \{0\}.
\end{equation*}
So, the result holds when $\ee \alpha \ff \in \ee J^2 \ff$. Now assume that $\ee \alpha \ff \in \ee J \ff \setminus \ee J^2 \ff$. By \ref{def: counterex}(4c), $\beta$ and $\ee \alpha \ff$ are equivalent modulo $\ee J^2 \ff$, so $\beta = \ee \alpha \ff + \gamma$ for some $\gamma \in \ee J^2 \ff$. Note that $\gamma \alpha^{n-2} = 0$. So,
\begin{equation*}
\beta \alpha^{n-2} = \ee \alpha \ff \alpha^{n-2} + \gamma \alpha^{n-2} = \ee \alpha \ff \alpha^{n-2} = \alpha^{n-1},
\end{equation*}
as required.
\end{proof}

\begin{Thm}\label{thm: counterex}
Let $R$ be a counterexample ring. Then, $\mcN(R)$ is not a two-sided ideal of $R[x]$.
\end{Thm}
\begin{proof}
Let $\beta \in \ee J \ff \setminus \ee J^2 \ff$. Form the following three polynomials:
\begin{equation*}
g(x) = (x^2+x)^{n-1}, \quad h(x) = \beta(x^2+x)^{n-2}, \quad \text{ and } \quad f(x) = g(x)-h(x).
\end{equation*}
Since $x^2 + x \in \mcN(R/J)$, we have $a^2+a \in J$ for all $a \in R$. By Lemma \ref{lem: counterex lem}(2), $g(a) = h(a)$ for all $a \in R$. Hence, $f \in \mcN(R)$. However, $f \ee \notin \mcN(R)$. To see this, note that $\beta \ee = 0$, so $h \ee = 0$ and $f \ee = g \ee$. By \ref{def: counterex}(3), there exists $r \in R$ such that $g(r) \ne 0$. Since $g(r) \in J^{n-1} = \ee J^{n-1} \ff$, we have $\ee g(r) = g(r) \ne 0$. Finally, $g \ee = \ee g$ because the coefficients of $g$ are central. Thus,
\begin{equation*}
(f\ee)(r) = (g\ee)(r) = (\ee g)(r) = \ee g(r) \ne 0.
\end{equation*}
We conclude that $f \ee \notin \mcN(R)$, and $\mcN(R)$ is not two-sided.
\end{proof}

\begin{Rem}\label{rem: counterex}
Let $\mcR$ be as in Example \ref{ex: Havlovec}. Then, $\mcR$ is a counterexample ring; one may check this directly, or apply Theorem \ref{thm: R_n rings} below. The polynomial $f$ constructed in Theorem \ref{thm: counterex} to show that $\mcN(\mcR)$ is not two-sided is slightly simpler than the analogous polynomial given in \cite{Havlovec1}.  Take
\begin{equation*}
\ee = \begin{pmatrix} 1&0&0&0\\0&0&0&0\\0&0&1&0\\0&0&0&0 \end{pmatrix} \quad \text{ and } \quad \beta = \begin{pmatrix} 0&1&0&0\\0&0&0&0\\0&0&0&1\\0&0&0&0 \end{pmatrix}.
\end{equation*}
Then, the polynomial from \cite{Havlovec1} is
\begin{align*}
\ee x^6 + \ee x^5 + (\ee + \beta) x^4 + \ee x^3 + \beta x^2 
&= \ee (x^6+x^5+x^4+x^3) + \beta(x^4+x^2)\\
&= \ee (x^2+x)^3 + \beta (x^2+x)^2,
\end{align*}
whereas the polynomial from Theorem \ref{thm: counterex} is $(x^2+x)^3 + \beta (x^2+x)^2$.
\end{Rem}

We close the paper by constructing, for each $n \geq 4$, a subring of $M_n(\F_2)$ for which the Jacobson radical has nilpotency $n$ and the null ideal is not two-sided. When $n=4$, the resulting ring is $\mcR$.

\begin{Def}\label{def: bigger rings}
For each $n \geq 4$, we define a collection $\mcR_n$ of matrices in $M_n(\F_2)$. In what follows, an asterisk $\ast$ denotes an entry of a matrix that has no restriction. That is, an $\ast$ entry can be either 0 or 1, and does not depend on any other entry of the matrix.

When $n$ is even, the matrices in $\mcR_n$ are $(n/2) \times (n/2)$ block matrices, where each block is $2 \times 2$. The diagonal blocks come from $T_2(\F_2)$, and matrices in $\mcR_n$ have identical diagonal blocks. Blocks below the diagonal are all 0, and blocks above the diagonal have no restrictions. So, when $n$ is even, the matrices in $\mcR_n$ have the form
\begin{equation}\label{eq: n even form}
\left(\begin{array}{cc|cc|c|cc}
a & b & \ast & \ast & \cdots & \ast & \ast\\
0 & c & \ast & \ast & \ddots & \ast & \ast\\
\hline
0 & 0 & a    & b    & \ddots & \ast & \ast\\
0 & 0 & 0    & c    & \ddots & \ast & \ast\\
\hline
0 & 0 & 0    & 0    & \ddots & a    & b   \\
0 & 0 & 0    & 0    & \cdots & 0    & c   \\
\end{array}\right)
\end{equation}
where $a, b, c \in \F_2$.

When $n$ is odd, write $n=2k+1$ for some $k \geq 2$. Each matrix in $\mcR_n$ has the form
\begin{equation*}
\begin{pmatrix} A & \ast \\ 0 & c \end{pmatrix},
\end{equation*}
where $A \in \mcR_{2k}$ and $c$ is the $(2k,2k)$-entry of $A$. Equivalently, matrices in $\mcR_n$ have the form
\begin{equation}\label{eq: n odd form}
\left(\begin{array}{cc|cc|c|cc|c}
a & b & \ast & \ast & \cdots & \ast & \ast & \ast\\
0 & c & \ast & \ast & \ddots & \ast & \ast & \ast\\
\hline
0 & 0 & a    & b    & \ddots & \ast & \ast & \ast\\
0 & 0 & 0    & c    & \ddots & \ast & \ast & \ast\\
\hline
0 & 0 & 0    & 0    & \ddots & a    & b    & \ast\\
0 & 0 & 0    & 0    & \ddots & 0    & c    & \ast\\
\hline
0 & 0 & 0    & 0    & \cdots & 0    & 0    & c\\
\end{array}\right)
\end{equation}
where $a, b, c \in \F_2$.

For the rings $\mcR_n$, we fix choices for $\ee$ and $\ff$. 
Following either \eqref{eq: n even form} or \eqref{eq: n odd form} (depending on whether $n$ is even or odd), let $\ee$ be the matrix with $a=1$ and all other entries equal to 0. This forces $\ff$ to be the matrix with $c=1$ and 0 in all other entries.
\end{Def}

It is straightforward to check that each set $\mcR_n$ from Definition \ref{def: bigger rings} is a ring. We will show that $\mcR_n$ is a counterexample ring when $n$ is even. When $n$ is odd, $\mcR_n$ does not meet the definition of a counterexample ring because $\ee J^{n-2} \ff$ contains matrices with nonzero entries in position $(1,n-1)$. So, $\ee J^{n-2} \ff \not\subseteq \ee J^{n-1} \ff$ and \ref{def: counterex}(4b) does not hold. Nevertheless, all of the other conditions in Definition \ref{def: counterex} still apply to $\mcR_n$, and we will show that Lemma \ref{lem: counterex lem}(2) is true for the ring. This is enough to prove that $\mcN(\mcR_n)$ is not two-sided when $n$ is odd. 

\begin{Lem}\label{lem: R_n has property (3)}
Let $n \geq 4$. Then, there exists $r_n \in \msJ(\mcR_n)$ such that $(r_n^2+r_n)^{n-1} \ne 0$.
\end{Lem}
\begin{proof}
For any value of $n$, the $n \times n$ matrix $\alpha_n$ with 1 along the first superdiagonal and 0 elsewhere is such that $\alpha_n^{n-1}$ has 1 in the $(1,n)$-entry and 0 elsewhere. In particular, $\alpha_n^{n-1} \ne 0$. For each $n$, $\alpha_n$ is an element of $\mcR_n$, so it suffices to show that there exists $r_n \in \msJ(\mcR_n)$ such that $r_n^2+r_n = \alpha_n$. We will define $r_n$ recursively.

When $n=4$, we can take
\begin{equation*}
r_4 = \begin{pmatrix} 0&1&1&0\\0&0&1&1\\0&0&0&1\\0&0&0&0 \end{pmatrix}, \quad \text{ for which } \quad r_4^2+r_4 = \begin{pmatrix} 0&1&0&0\\0&0&1&0\\0&0&0&1\\0&0&0&0 \end{pmatrix}.
\end{equation*}
Assuming we have determined $r_n \in \msJ(\mcR_n)$, let $I_n$ be the $n \times n$ identity matrix and let $w_n$ be the $n \times 1$ matrix $w_n = (0 \, \cdots \, 0 \; 1)^T$. Since $r_n \in \msJ(\mcR_n)$, the matrix $r_n+I_n$ is invertible. Take $v_n = (r_n + I_n)^{-1} w_n$ and let $r_{n+1}$ be the $(n+1) \times (n+1)$ matrix
\begin{equation*}
r_{n+1} = \left(\begin{array}{c|c}
r_n & v_n\\ \hline 0 & 0
\end{array}\right).
\end{equation*}
Then, $r_{n+1} \in \msJ(\mcR_{n+1})$ and
\begin{equation*}
r_{n+1}^2 + r_{n+1} = 
\left(\begin{array}{c|c}
r_n^2 & r_n v_n\\ \hline 0 & 0
\end{array}\right) 
+ 
\left(\begin{array}{c|c}
r_n & v_n\\ \hline 0 & 0
\end{array}\right) 
=
\left(\begin{array}{c|c}
r_n^2 + r_n & r_n v_n + v_n\\ \hline 0 & 0
\end{array}\right).
\end{equation*}
By construction, $r_n^2+r_n = \alpha_n$ and $r_n v_n + v_n = w_n$. Thus, $r_{n+1}^2 + r_{n+1} = \alpha_{n+1}$.
\end{proof}

\begin{Lem}\label{lem: R_n is a counterex}
Let $n \geq 4$. If $n$ is even, then $\mcR_n$ is a counterexample ring. If $n$ is odd, then $\mcR_n$ satisfies each property in Definition \ref{def: counterex} except \ref{def: counterex}(4b).
\end{Lem}
\begin{proof}
We first examine the properties in Definition \ref{def: counterex} that apply to all the rings $\mcR_n$. Let $J = \msJ(\mcR_n)$, which is the set of all strictly upper triangular matrices in $\mcR_n$. The diagonal entries of a matrix in $\mcR_n$ are determined by the two parameters $a \in \F_2$ and $c \in \F_2$ in \eqref{eq: n even form} and \eqref{eq: n odd form}, so $\mcR_n/J \cong \F_2 \oplus \F_2$. Next, the nilpotency of $J$ is at most $n$ because $\mcR_n \subseteq T_n(\F_2)$, and the matrix $\alpha_n$ defined in the proof of Lemma \ref{lem: R_n has property (3)} demonstrates that $J^{n-1} \ne \{0\}$. Thus, the nilpotency of $J$ is exactly $n$. This shows that \ref{def: counterex}(1) and \ref{def: counterex}(2) hold for $\mcR_n$, and \ref{def: counterex}(3) is satisfied by Lemma \ref{lem: R_n has property (3)}.

Now, if $\alpha \in J^{n-1}$, then the $(1,n)$-entry of $\alpha$ is either 0 or 1, and all other entries of $\alpha$ are 0. So, $\ee \alpha \ff = \alpha$ and $J^{n-1} = \ee J^{n-1} \ff$. Similarly, if $\alpha \in J^{n-2}$ then the only possible nonzero entries of $\alpha$ occur in positions $(1,n-1)$, $(1,n)$, and $(2,n)$. For such an $\alpha$, if $n$ is even then $(1,n-1)$-entry and $(2,n)$-entry of $\ee \alpha \ff$ are both 0. So, $\ee \alpha \ff \in \ee J^{n-1} \ff$ when $n$ is even, although this may fail when $n$ is odd, as discussed after Definition \ref{def: bigger rings}. Finally, regardless of the parity of $n$, the coset representatives of $\ee J \ff / \ee J^2 \ff$ correspond to the possible choices for $b$ in \eqref{eq: n even form} and \eqref{eq: n odd form}. Since $b \in \F_2$, we have $|\ee J \ff / \ee J^2 \ff| = 2$. So, $\mcR_n$ is a counterexample ring when $n$ is even, and if $n$ is odd then properties \ref{def: counterex}(4a) and \ref{def: counterex}(4c) still hold for $\mcR_n$.
\end{proof}

\begin{Lem}\label{lem: n odd counterex}
Let $n \geq 5$ be odd and let $J = \msJ(\mcR_n)$. Then, there exists $\beta \in \ee J \ff \setminus \ee J^2 \ff$ such that $\beta \alpha^{n-2} = \alpha^{n-1}$ for all $\alpha \in J$.
\end{Lem}
\begin{proof}
As noted in Definition \ref{def: bigger rings}, any matrix in $\mcR_n$ can be written as
\begin{equation*}
\begin{pmatrix} A & v \\ 0 & c \end{pmatrix},
\end{equation*}
where $A \in \mcR_{n-1}$, $v$ is an $(n-1) \times 1$ matrix, and $c$ is equal to the $(n-1,n-1)$-entry of $A$. Since $\mcR_{n-1}$ is a counterexample ring by Lemma \ref{lem: R_n is a counterex}, we may apply Lemma \ref{lem: counterex lem} to find $\beta_1 \in \ee\msJ(\mcR_{n-1})\ff$ such that $\beta_1 \alpha_1^{n-3} = \alpha_1^{n-2}$ for all $\alpha_1 \in \msJ(\mcR_{n-1})$. Take
\begin{equation*}
\beta = \begin{pmatrix} \beta_1 & 0 \\ 0 & 0 \end{pmatrix} \in \mcR_n.
\end{equation*}
Then, $\beta \in \ee J \ff \setminus \ee J^2 \ff$. Let $\alpha \in J$ and write
\begin{equation*}
\alpha = \begin{pmatrix} \alpha_1 & v \\ 0 & 0 \end{pmatrix},
\end{equation*}
where $\alpha_1 \in \msJ(\mcR_{n-1})$ and $v$ is $(n-1) \times 1$. Note that $\alpha_1^{n-1} = 0$ and $\beta_1 \alpha_1^{n-2} = 0$ because $\msJ(\mcR_{n-1})$ has nilpotency $n-1$. So,
\begin{equation*}
\alpha^{n-1} = \begin{pmatrix} \alpha_1^{n-1} & \alpha_1^{n-2} v \\ 0 & 0 \end{pmatrix} =  \begin{pmatrix} 0 & \alpha_1^{n-2} v \\ 0 & 0 \end{pmatrix},
\end{equation*}
whereas
\begin{equation*}
\beta \alpha^{n-2} = \begin{pmatrix} \beta_1 & 0 \\ 0 & 0 \end{pmatrix} \begin{pmatrix} \alpha_1^{n-2} & \alpha_1^{n-3} v \\ 0 & 0 \end{pmatrix} =  \begin{pmatrix} \beta_1\alpha_1^{n-2} & \beta_1\alpha_1^{n-3} v \\ 0 & 0 \end{pmatrix} = \begin{pmatrix} 0 & \alpha_1^{n-2} v \\ 0 & 0 \end{pmatrix}.
\end{equation*}
Thus, $\beta \alpha^{n-2} = \alpha^{n-1}$ for all $\alpha \in J$, as desired.
\end{proof}

\begin{Thm}\label{thm: R_n rings}
Let $n \geq 4$. Then, $\mcN(\mcR_n)$ is not a two-sided ideal of $\mcR_n[x]$.
\end{Thm}
\begin{proof}
When $n$ is even, this follows from Lemma \ref{lem: R_n is a counterex} and Theorem \ref{thm: counterex}. When $n$ is odd, let $J = \msJ(\mcR_n)$. By Lemma \ref{lem: n odd counterex} there exists $\beta \in \ee J \ff \setminus \ee J^2 \ff$ such that $\beta \alpha^{n-2} = \alpha^{n-1}$ for all $\alpha \in J$. The existence of this $\beta$, along with the properties of counterexample rings that $\mcR_n$ does exhibit, are enough to construct the polynomials $f$, $g$, and $h$ from the proof of Theorem \ref{thm: counterex} and conclude that $\mcN(\mcR_n)$ is not two-sided when $n$ is odd.
\end{proof}

Each of the rings $\mcR_n$ has $\mcR=\mcR_4$ as a residue ring. Indeed, let $\mcI_n \subseteq \mcR_n$ be the set of all matrices in $\mcR_n$ with 0 in every entry in columns 1--4. Then, $\mcI_n$ is a two-sided ideal of $\mcR_n$ and $\mcR_n/\mcI_n \cong \mcR_4$. This observation prompts one final question for further research.

\begin{Ques}\label{ques: modding out}
Does there exist a finite ring $R$ such that $\mcN(R)$ is not two-sided and $R$ does not have $\mcR_4$ as a residue ring?
\end{Ques}

\section*{AI Use Statement}
No artificial intelligence (AI), large language model (LLM), or similar tool was used in this project, nor in the preparation of this paper.

\bibliographystyle{plainurl}
\bibliography{Index_arXiv}
 
\end{document}